\documentclass[12pt]{article}
\usepackage{amsmath,amssymb,amsthm,cite}
\usepackage{hyperref}

\newtheorem{theorem}{Theorem}[section]
\newtheorem{thmy}{Theorem}

\def\barr{\begin{array}}
\def\earr{\end{array}}

\title{Completely realisable groups}
\author{Georgiana Fasol\u a and Marius T\u arn\u auceanu}
\date{March 15, 2023}

\begin{document}

\maketitle

\begin{abstract}
Given a construction $f$ on groups, we say that a group $G$ is \textit{$f$-realisable} if there is a group $H$ such that $G\cong f(H)$, and \textit{completely $f$-realisable} if there is a group $H$ such that $G\cong f(H)$ and every subgroup of $G$ is isomorphic to $f(H_1)$ for some
subgroup $H_1$ of $H$ and vice versa.

In this paper, we determine completely ${\rm Aut}$-realisable groups. We also study $f$-realisable groups for $f=Z,F,M,D,\Phi$, where $Z(H)$, $F(H)$, $M(H)$, $D(H)$ and $\Phi(H)$ denote the center, the Fitting subgroup, the Chermak-Delgado subgroup, the derived subgroup and the Frattini subgroup of the group $H$, respectively.
\end{abstract}

{\small
\noindent
{\bf MSC2020\,:} Primary 20D30; Secondary 20D45, 20D25.

\noindent
{\bf Key words\,:} inverse group theory, (completely) $f$-realisable groups, automorphism groups, integrals of groups.}

\section{Introduction}

In group theory, there are many constructions $f$ which start from a group $H$ and produce another group $f(H)$. Examples of such group-theoretical constructions are: center, central quotient, derived quotient, Frattini subgroup, Fitting subgroup, Chermak-Delgado subgroup, automorphism group, Schur multiplier, other cohomology groups, and various constructions from permutation groups. For each of these constructions, there is an inverse problem:
\begin{equation}
\mbox{Given a group } G, \mbox{ is there a group } H \mbox{ such that } G\cong f(H)?
\end{equation}

Several new results related to this problem have been obtained in \cite{1,2,6,9,10} for $f(H)=D(H)$, the derived subgroup of $H$. Note that in these papers the group $H$ with the property $G\cong D(H)$ has been called an \textit{integral} of $G$ by analogy with calculus. Moreover, we recall Problem 10.19 in \cite{1} that asks to classify the groups in which all subgroups are integrable. It constitutes the starting point for our discussion.

Other results of the same type are given by \cite{7,12,20,21,22} for $f(H)=\Phi(H)$, the Frattini subgroup of $H$. In this case, there is a precise cha\-rac\-te\-ri\-za\-tion of finite groups $G$ for which (1) has solutions, namely
\begin{equation}
G\cong\Phi(H) \mbox{ for some group } H \mbox{ if and only if } {\rm Inn}(G)\subseteq\Phi({\rm Aut}(G))\nonumber
\end{equation}(see \cite{7}).

The same problem for $f(H)={\rm Aut}(H)$, the automorphism group of $H$, has been studied in \cite{16,17}. We also recall the well-known class of \textit{capable groups}, i.e. the groups $G$ such that (1) has solutions for $f(H)={\rm Inn}(H)$, the inner automorphism group of $H$. Their study was initiated by R. Baer \cite{3} and continued in many other papers (see e.g. \cite{5,8}).

Inspired by these studies, we introduce the following two notions.

Given a construction $f$ on groups, we say that a group $G$ is
\begin{description}
\item[\hspace{10mm}{\rm a)}]\textit{$f$-realisable} if there is a group $H$ such that $G\cong f(H)$
\end{description}and
\begin{description}
\item[\hspace{10mm}{\rm b)}]\textit{completely $f$-realisable} if there is a group $H$ such that:
\begin{itemize}
\item[{\rm i)}] $G\cong f(H)$;
\item[{\rm ii)}] $\forall\, G_1\leq G, \exists\, H_1\leq H$ such that $G_1\cong f(H_1)$;
\item[{\rm iii)}] $\forall\, H_1\leq H, \exists\, G_1\leq G$ such that $f(H_1)\cong G_1$.
\end{itemize}
\end{description}Clearly, if $f$ is monotone, then i) follows from ii) and iii). Clearly, any subgroup of an $f$-realisable group is itself $f$-realisable. This holds, for example, for $f=D$. Also, we observe that completely $D$-realisable groups\footnote{Another suitable name for these groups would be \textit{completely integrable groups}. For such a group $G$, a group $H$ satisfying i)-iii) is called a \textit{complete integral} of $G$.} are solutions of Problem 10.19 in \cite{1}.

Throughout this paper, we assume that the above groups $G$ and $H$ are both finite. In Section 2 we will determine completely ${\rm Aut}$-realisable groups, while in Section 3 we will present several results concerning completely $f$-realisable groups for $f=Z,F,M,D,\Phi$, where $Z(H)$, $F(H)$, $D(H)$ and $\Phi(H)$ denote the center, the Fitting subgroup, the Chermak-Delgado subgroup, the derived subgroup and the Frattini subgroup of the group $H$, respectively.

Most of our notation is standard and will usually not be repeated here. Elementary notions and results on groups can be found in \cite{13,19}.
For subgroup lattice concepts we refer the reader to \cite{18}.

\section{Completely ${\rm Aut}$-realisable groups}

First of all, we recall some results of MacHale \cite{16,17} concerning groups $H$ with a particular automorphism group.

\begin{theorem}
The following hold:
\begin{itemize}
\item[{\rm (a)}] There is no group $H$ such that ${\rm Aut}(H)\cong\mathbb{Z}_m$ for any odd number $m>1$,
\item[{\rm (b)}] There exists a group $H$ such that ${\rm Aut}(H)\cong\mathbb{Z}_{p^2}$ for a prime $p$ if and only if $p=2$ and $H\cong\mathbb{Z}_5$ or $H\cong\mathbb{Z}_{10}$,
\item[{\rm (c)}] There exists a group $H$ such that ${\rm Aut}(H)\cong\mathbb{Z}_p^3$ for a prime $p$ if and only if $p=2$ and $H\cong\mathbb{Z}_{24}$,
\item[{\rm (d)}] There exists a group $H$ such that ${\rm Aut}(H)\cong\mathbb{Z}_p\times\mathbb{Z}_{p^2}$ for a prime $p$ if and only if $p=2$ and $H\cong\mathbb{Z}_{15}$ or $H\cong\mathbb{Z}_{20}$ or $H\cong\mathbb{Z}_{30}$,
\item[{\rm (e)}] There is no group $H$ such that ${\rm Aut}(H)\cong\mathbb{Z}_2^4$.
\end{itemize}
\end{theorem}

\begin{proof}
Parts (a) and (b), respectively, follow from parts (i) and (iv)(c) of Theorem 1 in \cite{16}. Parts (c) and (d) are given by parts (iii) and (iv) of Theorem 2 in \cite{16}. Part (e) is a summarized version of Theorem 2 from \cite{17}.
\end{proof}

In \cite{14}, the solutions $H$ of ${\rm Aut}(H)\cong G$ have been also determined for other important classes of groups $G$ (see e.g. Theorem 4.2 for $G=A_n$, Theorem 4.4 for $G=S_n$ or Theorem 6.3 for $G=D_{2n}$). Also, we point out another interesting result of Ledermann and Neumann \cite{15} which states that for every $n>0$, there exists a bound $f(n)$ such that if $G$ is a finite group with $|G|\geq f(n)$, then $|{\rm Aut}(G)|\geq n$.

We are now able to give a description of completely ${\rm Aut}$-realisable groups.

\begin{theorem}
A group is completely ${\rm Aut}$-realisable if and only if it is an elementary abelian $2$-group of rank at most $3$.
\end{theorem}

\begin{proof}
Let $G$ be a completely ${\rm Aut}$-realisable group. Then there is a group $H$ such that:
\begin{itemize}
\item[{\rm i)}] $G\cong{\rm Aut}(H)$;
\item[{\rm ii)}] $\forall\, G_1\leq G, \exists\, H_1\leq H$ such that $G_1\cong{\rm Aut}(H_1)$;
\item[{\rm iii)}] $\forall\, H_1\leq H, \exists\, G_1\leq G$ such that ${\rm Aut}(H_1)\cong G_1$.
\end{itemize}Since the groups $\mathbb{Z}_p$ with $p$ an odd prime are not ${\rm Aut}$-realisable by Theorem 2.1 (a), it follows that $G\cong{\rm Aut}(H)$ is a $2$-group. Then so is ${\rm Inn}(H)\cong\frac{H}{Z(H)}$\,. Let $$|H|=p_1^{n_1}\cdots p_k^{n_k},$$where $p_1=2$ and $p_2$, ..., $p_k$ are odd primes, and denote by $P_i$ a Sylow $p_i$-subgroup of $H$, $\forall\, i=1,\dots,k$. Then $P_2,\dots, P_k\subseteq Z(H)$ and therefore
\begin{equation}
H=A\rtimes P_1, \mbox{ where } A=\prod_{i=2}^k P_i.\nonumber
\end{equation}On the other hand, we have $n_2=\dots=n_k=1$ because $p^2\mid |H|$ implies $p\mid |{\rm Aut}(H)|$ for any prime $p$ (see e.g. \cite{11}). Consequently,
\begin{equation}
A\cong\mathbb{Z}_{p_2\cdots p_k}.\nonumber
\end{equation}

Assume that $H$ is not the direct product of $A$ and $P_1$. Then $H$ contains a subgroup $H_1\cong D_{2p_j}$ for some $j=2,\dots,k$. By iii), there exists $G_1\leq G$ such that ${\rm Aut}(H_1)\cong G_1$. This implies that $|{\rm Aut}(H_1)|=p_j(p_j-1)$ divides $|G|$ and so $p_j$ divides $|G|$, a contradiction. Thus we have $H=A\times P_1$.

Since $A$ and $P_1$ are of coprime orders, we get
\begin{equation}
G\cong{\rm Aut}(H)\cong{\rm Aut}(A)\times{\rm Aut}(P_1)\cong\left(\prod_{i=2}^k\mathbb{Z}_{p_i}^{\times}\right)\times{\rm Aut}(P_1),\nonumber
\end{equation}which shows that
\begin{equation}
|G|=\prod_{i=2}^k(p_i-1)|{\rm Aut}(P_1)|.\nonumber
\end{equation}Since $G$ is a $2$-group, then $p_2,\dots,p_k$ are Fermat primes\footnote{Recall that a Fermat prime is a prime number of the form $2^{2^n}+1$, $n\in\mathbb{N}$.} and $P_1$ is a $2$-group whose auto\-morphism group is also a $2$-group.

Let $K$ be an abelian subgroup of $P_1$. If $K$ is not cyclic, then it contains a subgroup $K_1\cong\mathbb{Z}_2\times\mathbb{Z}_2$. It follows that ${\rm Aut}(K_1)$ is isomorphic to a subgroup of $G$ and therefore $|{\rm Aut}(K_1)|=6$ is a power of $2$, a contradiction. Thus all abelian subgroups of $P_1$ are cyclic, implying that $P_1$ is either cyclic or a generalized quaternion $2$-group (see e.g. (4.4) of \cite{19}, II). Since $Q_{2^{n_1}}$ possesses subgroups isomorphic to $Q_8$ and $|{\rm Aut}(Q_8)|=24$ is not a power of $2$, we infer that $P_1$ is cyclic, i.e. $P_1\cong\mathbb{Z}_{2^{n_1}}$. Then
\begin{equation}
H\cong\mathbb{Z}_{p_2\cdots p_k}\times\mathbb{Z}_{2^{n_1}}\cong\mathbb{Z}_{2^{n_1}p_2\cdots p_k}\nonumber
\end{equation}and so
\begin{equation}
G\cong\mathbb{Z}_{2^{n_1}p_2\cdots p_k}^{\times}.\nonumber
\end{equation}

Note that $\mathbb{Z}_{2^{n_1}}^{\times}$ is trivial for $n_1=0,1$ and $\mathbb{Z}_{2^{n_1}}^{\times}\cong\mathbb{Z}_2\times\mathbb{Z}_{2^{n_1-2}}$ for $n_1\geq 2$.

Assume first that $k=1$, i.e. $H\cong\mathbb{Z}_{2^{n_1}}$ and $G\cong\mathbb{Z}_{2^{n_1}}^{\times}$. If $n_1\geq 4$, then $G$ contains a subgroup $G_1\cong\mathbb{Z}_{2^2}$
and so, by Theorem 2.1 (b), we know that $G_1\cong{\rm Aut}(H_1)$ where $H_1\cong\mathbb{Z}_5$ or $H_1\cong\mathbb{Z}_{10}$. In both these cases, we have that $5$ divides
the order of $H_1$ and thus $5$ divides the order of $H$, which is impossible because $H$ is a $2$-group. Thus $n_1\leq 3$ and we get $G\cong\mathbb{Z}_2^r$, where $r=0,1,2$.

Assume now that $k\geq 3$. Then $G$ contains a subgroup $G_1\cong\mathbb{Z}_2^4$. It follows that $G_1\cong{\rm Aut}(H_1)$ for some subgroup $H_1\leq H$, contradicting Theorem 2.1 (e). Thus $k=2$ and by Theorem 2.1 (c) we have that $3$ divides $|H|$ and so $p_2=3$. If $n_1\geq 4$, then $G$ contains a subgroup $G_1\cong\mathbb{Z}_2\times\mathbb{Z}_{2^2}$ and Theorem 2.1 (d) implies that $5$ divides $|H|$, contradicting the fact that $H$ is a $2$-group. Hence $n_1\leq 3$, leading to
$G\cong\mathbb{Z}_2^r$ with $r=1,2,3$.

Conversely, if $G\cong\mathbb{Z}_2^r$ with $r\leq 3$, then it suffices to take $H=\mathbb{Z}_2$ for $r=0$, $H=\mathbb{Z}_4$ for $r=1$, $H=\mathbb{Z}_8$ for $r=2$ and $H=\mathbb{Z}_8\times\mathbb{Z}_3$ for $r=3$. This completes the proof.
\end{proof}

\section{Completely $f$-realisable groups, where $f=Z,F,M,D,\Phi$}

The problem of determining completely $f$-realisable groups for $f=Z$ and $f=F$ is trivial, namely:

\begin{itemize}
\item[$\bullet$] A group is completely $Z$-realisable if and only if it is abelian.
\item[$\bullet$] A group is completely $F$-realisable if and only if it is nilpotent.
\end{itemize}

The same thing can be also said for $f=M$. Recall that, given a finite group $H$, the \textit{Chermak-Delgado measure} of a subgroup $K$ of $H$ is defined by
\begin{equation}
m_H(K)=|K||C_H(K)|.\nonumber
\end{equation}Let
\begin{equation}
m^*(H)={\rm max}\{m_H(K)\mid K\leq H\} \mbox{ and } {\cal CD}(H)=\{K\leq G\mid m_H(K)=m^*(H)\}.\nonumber
\end{equation}Then the set ${\cal CD}(H)$ forms a modular, self-dual sublattice of the subgroup lattice of $H$, which is called the \textit{Chermak-Delgado lattice} of $H$ (see Theorem 1.44 in  \cite{13}). The minimal member $M(H)$ of ${\cal CD}(H)$ is called the \textit{Chermak-Delgado subgroup} of $H$. Note that $M(H)$ is characteristic, abelian and contains $Z(H)$ by Corollary 1.45 in \cite{13}. So, a (completely) $M$-realisable group is abelian. Conversely, it is clear that every subgroup of an abelian group is its own Chermak-Delgado subgroup. Thus we have:

\begin{itemize}
\item[$\bullet$] A group is completely $M$-realisable if and only if it is abelian.
\end{itemize}

In what follows we will focus on the other two cases. Recall that a group $G$ is completely $D$-realisable if there is a group $H$ such that, up to isomorphism, we have
\begin{equation}
L(G)=\{H_1'\mid H_1\leq H\},
\end{equation}where $L(G)$ denotes the subgroup lattice of $G$. An important class of completely $D$-realisable groups is given by the following theorem.

\begin{theorem}
All abelian groups are completely $D$-realisable.
\end{theorem}

\begin{proof}
Guralnick \cite{10} showed that if $A$ is an abelian group of order $n$, then the group $H=A\wr\mathbb{Z}_2$ is an integral of $A$. Clearly, we have $H_1'\leq A$, for all $H_1\leq H$. Conversely, we observe that if $A_1\leq A$, then $H$ contains a subgroup $H_1\cong A_1\wr\mathbb{Z}_2$ which is an integral of $A_1$.
\end{proof}

Note that there exist completely $D$-realisable non-abelian groups, e.g. $G=A_4$ for which it suffices to take $H=S_4$, and even completely $D$-realisable non-abelian $p$-groups, e.g. $G={\rm He}_3$, the Heisenberg group of order $3^3$, for which it suffices to take $H=SmallGroup(216,36)$.
\bigskip

\noindent{\bf Remark.} We conjecture that $A_4$ is the smallest completely $D$-realisable non-abelian group. The dihedral groups $D_{2n}$ with $n=3,4,5,6$ and the dicyclic group ${\rm Dic}_3$ are not $D$-realisable because each of them has a characteristic cyclic subgroup which is not contained in center (see Proposition 3.1 of \cite{1}). The quaternion group $Q_8$ is $D$-realisable, a group $H$ with $Q_8\cong D(H)$ being necessarily a proper semidirect product $P\rtimes A$, where $P$ is a $2$-group con\-ta\-i\-ning a normal subgroup isomorphic to $Q_8$ and having $D(P)$ cyclic of order $2$ or $4$, and $A$ is an abelian group of odd order\footnote{The smallest integral of $Q_8$ is ${\rm SL}(2,3)\cong Q_8\rtimes\mathbb{Z}_3$ (see \cite{2}).}.

Indeed, if $H$ is a finite group with $Q_8\cong D(H)$ and $P$ is a Sylow $2$-subgroup of $H$ including $D(H)$, then $P$ is normal in $H$ and $H/P$ is abelian. Since $P$ and $H/P$ are of coprime orders, the Schur-Zassenhaus theorem leads to $H\cong P\rtimes A$, where $A\cong H/P$ is abelian of odd order. Moreover, $D(P)$ is a proper subgroup of $D(H)$ because $Q_8$ is not $p$-integrable (see Theorem 4.2 of \cite{2}) and so it is cyclic of order $2$ or $4$.

We note that a GAP search over all groups of order less or equal that $500$ gives no complete integral of $Q_8$.

We also note that the class of completely $D$-realisable groups is properly contained in the class of $D$-realisable groups: $A_5$ is $D$-realisable - we have $A_5=D(S_5)$, but not completely $D$-realisable - it has a subgroup isomorphic to $D_{10}$ which is not $D$-realisable. Since $A_n$ have subgroups of type $A_5$, for all $n\geq 5$, we get:

\begin{theorem}
Alternating groups $A_n$ with $n\geq 5$ are $D$-realisable, but not completely $D$-realisable.
\end{theorem}

Next we will focus on completely $\Phi$-realisable groups. Recall that such a group $G$ is nilpotent and satisfies ${\rm Inn}(G)\subseteq\Phi({\rm Aut}(G))$. Again, we have a result similar with Theorem 3.1.

\begin{theorem}
All abelian groups are completely $\Phi$-realisable.
\end{theorem}

\begin{proof}
Given an abelian group $G$, we have to prove that there exists a group $H$ such that:
\begin{itemize}
\item[{\rm i)}] $G\cong\Phi(H)$;
\item[{\rm ii)}] $\forall\, G_1\leq G, \exists\, H_1\leq H$ such that $G_1\cong\Phi(H_1)$;
\item[{\rm iii)}] $\forall\, H_1\leq H, \exists\, G_1\leq G$ such that $\Phi(H_1)\cong G_1$.
\end{itemize}Since $\Phi$ is completely multiplicative, that is
\begin{equation}
\hspace{10mm}\Phi(\prod_{i=1}^m G_i)\cong\prod_{i=1}^m \Phi(G_i) \mbox{ for all groups } G_i,\, i=1,...,m,\nonumber
\end{equation}it suffices to assume that $G$ is an abelian $p$-group and to prove that there exists a $p$-group $H$ with the above properties. This is clear because for $G\cong\prod_{i=1}^k\mathbb{Z}_{p^{n_i}}$ we can choose $H\cong\prod_{i=1}^k\mathbb{Z}_{p^{n_i+1}}$.
\end{proof}

Note that all $\Phi$-realisable $p$-groups of order $p^3$ are abelian (see e.g. Lemma 1 of \cite{12}). The same thing can be also said about $\Phi$-realisable $2$-groups of order $2^4$ (see e.g. Theorem 1 of \cite{20}). An example of a $\Phi$-realisable non-abelian $5$-group of order $5^4$ is presented in Remark 5 of \cite{4}:
\begin{equation}
G=\langle x,y,z,t\mid x^5=y^5=z^5=t^5=1, [z,t]=x, [x,t]=[y,t]=l\rangle\nonumber
\end{equation}for which we can choose
\begin{align*}
H=&\langle u,v,w,x,y,z\mid u^5=v^5=w^5=x^5=y^5=z^5=1, [v,w]=[v,x]=[v,z]\\
&=[x,y]=1, [v,y]=[x,w]=[w,y]=u, [w,z]=v, [x,z]=w, [y, z]=x\rangle.\nonumber
\end{align*}

Finally, we remark that there exist $\Phi$-realisable groups that are not completely $\Phi$-realisable, for example $G=\mathbb{Z}_2\times Q_8$ - we have $G=\Phi(H)$, where $H=SmallGroup(64,9)$, but $G$ contains a subgroup isomorphic to $Q_8$ which is the Frattini subgroup of no group.
\bigskip

We end this paper by proposing the following two open problems:
\bigskip

\noindent{\bf Problem 1.} Classify completely $D$-realisable and completely $\Phi$-realisable groups.
\bigskip

\noindent{\bf Problem 2.} Study (completely) $f$-realisable groups for other constructions $f$ on groups, e.g. when $f(H)$ is the Carter subgroup of the finite solvable group $H$.
\bigskip

\noindent{\bf Acknowledgement.} The authors are grateful to the reviewer for remarks which improve the previous version of the paper.

\vspace*{3ex}
\small

\begin{minipage}[t]{7cm}
Georgiana Fasol\u a \\
Faculty of  Mathematics \\
"Al.I. Cuza" University \\
Ia\c si, Romania \\
e-mail: \!{\tt georgiana.fasola@student.uaic.ro}
\end{minipage}
\hfill\hspace{20mm}
\begin{minipage}[t]{5cm}
Marius T\u arn\u auceanu \\
Faculty of  Mathematics \\
"Al.I. Cuza" University \\
Ia\c si, Romania \\
e-mail: \!{\tt tarnauc@uaic.ro}
\end{minipage}


\begin{thebibliography}{10}
\bibitem{1} J. Ara\'{u}jo, P.J. Cameron, C. Casolo and F. Matucci, {\it Integrals of groups}, Israel. J. Math. {\bf 234} (2019), 149-178.
\bibitem{2} J. Ara\'{u}jo, P.J. Cameron, C. Casolo, F. Matucci, C. Quadrelli, {\it Integrals of groups}, II, to appear in Israel. J. Math.
\bibitem{3} R. Baer, {\it Groups with preassigned central and central quotient group}, Trans. Amer. Math. Soc. {\bf 44} (1938), 378–412.
\bibitem{4} H. Bechtell, {\it Frattini subgroups and $\Phi$-central groups}, Pacific J. Math. {\bf 18} (1966), 15–23.
\bibitem{5} F.R. Beyl, U. Felgner and P. Schmid, {\it On groups occuring as centre factor groups}, J. Algebra {\bf 61} (1979), 161–177.
\bibitem{6} R. Blyth, F. Fumagalli, F. Matucci, {\it On some questions related to integrable groups}, to appear in Ann. Mat. Pura Appl.
\bibitem{7} B. Eick, {\it The converse of a theorem of W. Gasch\"{u}tz on Frattini subgroups}, Math. Z. {\bf 224} (1997), 103-111.
\bibitem{8} G. Ellis, {\it On the capability of groups}, Proc. Edinburgh Math. Soc. {\bf 41} (1998), 487-495.
\bibitem{9} K. Filom and B. Miraftab, {\it Integral of groups}, Comm. Algebra {\bf 45} (2017), 1105–1113.
\bibitem{10} R.M. Guralnick, {\it On groups with decomposable commutator subgroups}, Glasgow Math. J. {\bf 19} (1978), 159–162.
\bibitem{11} I.N. Herstein and J.E. Adney, {\it A note on the automorphism group of a finite group}, Amer. Math. Monthly {\bf 59} (1952), 309–310.
\bibitem{12} W. Mack Hill and Donald B. Parker, {\it The nilpotence class of the Frattini subgroup}, Israel J. Math. {\bf 15} (1973), 211–215.
\bibitem{13} I.M. Isaacs, {\it Finite group theory}, Amer. Math. Soc., Providence, R.I., 2008.
\bibitem{14} H.K. Iyer, {\it On solving the equation ${\rm Aut}(X)=G$}, Rocky Mountain J. Math. {\bf 9} (1979), 653-670.
\bibitem{15} W. Ledermann and B.H. Neumann, {\it On the order of the automorphism group of a finite group} I, Proc. Roy. Soc. London. Ser. A. {\bf 233} (1956), 494-506.
\bibitem{16} D. MacHale, {\it Some finite groups which are rarely automorphism groups}, I, Proc. Roy. Irish. Acad. Sect. A. {\bf 81A} (1981), 209–215.
\bibitem{17} D. MacHale, {\it Some finite groups which are rarely automorphism groups}, II, Proc. Roy. Irish. Acad. Sect. A. {\bf 83A} (1983), 189–196.
\bibitem{18} R. Schmidt, {\it Subgroup lattices of groups}, de Gruyter Expositions in Ma\-the\-ma\-tics 14, de Gruyter, Berlin, 1994.
\bibitem{19} M. Suzuki, {\it Group theory}, I, II, Springer Verlag, Berlin, 1982, 1986.
\bibitem{20} R.W. van der Waall, {\it Certain normal subgroups of the Frattini subgroup of a finite group}, Indag. Math. {\bf 77} (1974), 382-386.
\bibitem{21} R.W. van der Waall and C.H.W.M. de Nijs, {\it On the embedding of a finite group as Frattini subgroup}, Bull. Belg. Math. Soc. Simon Stevin  {\bf 2} (1995), 519-527.
\bibitem{22} C.R.B. Wright, {\it Frattini embeddings of normal subgroups}, Proc. Amer. Math. Soc. {\bf 78} (1980), 319–320.
\end{thebibliography}
\end{document}